\documentclass[11pt]{amsart}
\usepackage{graphicx}
\usepackage{amsmath,amsfonts,amssymb,latexsym,amsthm,amscd}
\usepackage{multirow}
 \newtheorem{theorem}{Theorem}[section]

 \newtheorem{remark}[theorem]{Remark}

\begin{document}

\title[Extended Mapping Class Group]{The Torsion Generating Set Of The Extended Mapping Class Groups In Low Genus Cases}

\author{Xiaoming Du}
\address{South China University of Technology,
  Guangzhou 510640, P.R.China}
\email{scxmdu@scut.edu.cn}



\thanks{}

\maketitle

\begin{abstract}
We prove that for genus $g=3,4$, the extended mapping class group $\text{Mod}^{\pm}(S_g)$ can be generated by two elements of finite orders.
But for $g=1$, $\text{Mod}^{\pm}(S_1)$ cannot be generated by two elements of finite orders.
\end{abstract}

\section{Introduction}

Korkmaz has proved that the mapping class group $\text{Mod}(S_g)$
can be generated by two elements of finite orders in \cite{Ko1}.
Using the notation that $\langle m, n \rangle$ (m, n are integers) to mean
a group can be generated by two elements whose orders are $m$ and $n$ respectively,
Korkmaz's result says:
\begin{table}[h]
\begin{center}
\begin{tabular}{|c|c|}
  \hline
  \multirow{2}{*}{$\text{Mod}(S_g)$} & torsion generating set \\
  & consisting of two elements\\
  \hline
  $g = 1$ & $\langle 4, 6 \rangle$ \\
  \hline
  $g = 2$ & $\langle 6, 10 \rangle$ \\
  \hline
  $g \geq 3$ & $\langle 4g+2, 4g+2 \rangle$ \\
  \hline
\end{tabular}
\end{center}
\end{table}

It is an open problem listed in \cite{Ko2} that whether the extended mapping class group $\text{Mod}^{\pm}(S_g)$ can be generated by two torsion elements.
In \cite{Du}, the author partially solved such a problem when the genus $g \geq 5$. In this paper, we deal with $g=1, 3, 4$.

When $g = 3, 4$, the method and idea in the process of calculation in this paper are mostly the same as those in \cite{Du} and \cite{Ko1}.
The reason for $g=3$ and $g=4$ should be treated separately is as the follow. When the genus is high,
there will be plenty of space to find a simple closed curve satisfying two conditions: (1) it lies in the periodic orbit;
(2) it does not intersect with some given curves. When the genus is less than 5, we cannot do this. So we use other treatment carefully.
When $g = 1$, we use the presentation of $GL(2,\mathbb{Z})$ to prove it cannot be generated by two elements of finite orders.
So we can summarize the result as follow:
\begin{table}[h]
\begin{center}
\begin{tabular}{|c|c|}
  \hline
  \multirow{2}{*}{$\text{Mod}{\pm}(S_g)$} & torsion generating set \\
  & consisting of two elements\\
  \hline
  $g = 1$ & impossible \\
  \hline
  $g = 2$ & still unknown \\
  \hline
  $g \geq 3$ & $\langle 2, 4g+2 \rangle$ \\
  \hline
\end{tabular}
\end{center}
\end{table}

\section{Preliminary}

\textbf{Notations.}

(a) We use the convention of functional notation, namely, elements of
the mapping class group are applied right to left, i.e. the composition $FG$ means
that $G$ is applied first.

(b) A Dehn twist means a right-hand Dehn twist.

(c) We denote the curves by lower case letters $a$, $b$, $c$, $d$
(possibly with subscripts) and the Dehn twists about them by the corresponding
capital letters $A$, $B$, $C$, $D$. Notationally we do not distinguish a
diffeomorphism/curve and its isotopy class.

\textbf{Humphries generators and the $(4g+2)$-gon.}

Humphries have proved the following theorem (\cite{Hu}).

\begin{theorem}
Let $a_1, a_2, \dots\, a_{2g}, b_0$ be the curves as on the left-hand side of figure 1.
Then the mapping class group $\text{Mod}(S_g)$ is generated by $A_i$'s and $B_0$.
\end{theorem}

\begin{figure}[htbp]
\centering
\includegraphics[height=3.5cm]{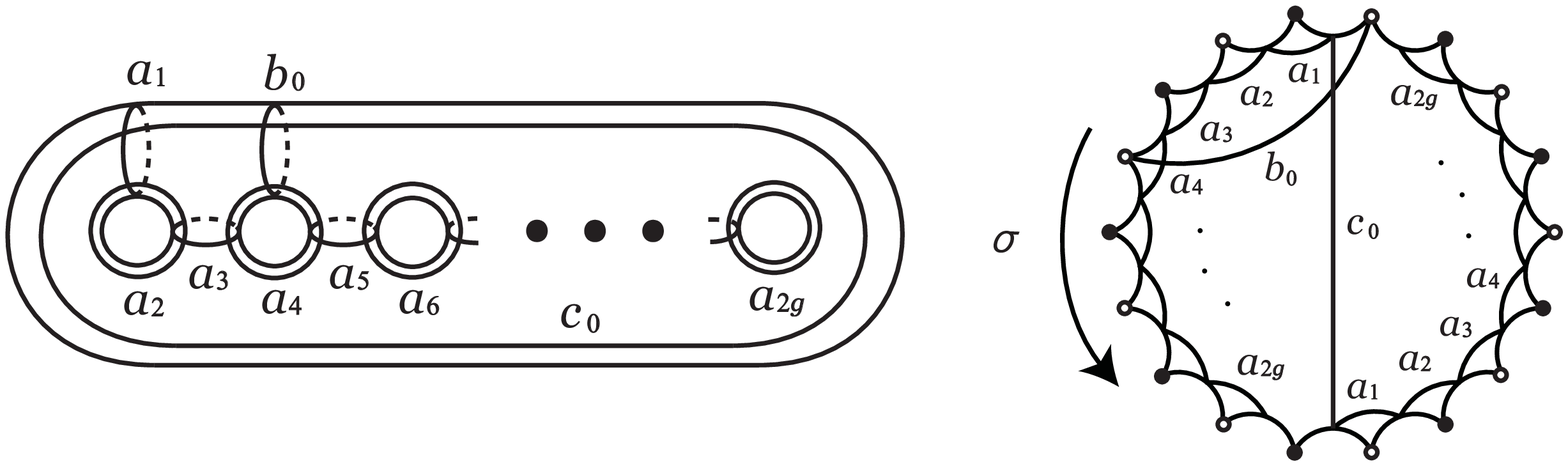} \\
\textbf{Figure 1.}
\end{figure}

The genus $g$ surface can be looked as a $(4g+2)$-gon, whose opposite edges are glued together in pairs.
$(4g+2)$ vertices of the $(4g+2)$-gon are glued to be two vertices.

We can also draw the curves $a_1, a_2, \dots\, a_{2g}, b_0$ on the $(4g+2)$-gon as the right-hand side of figure 1.
There is a natural rotation $\sigma$ of the $(4g+2)$-gon that sends $a_i$ to $a_{i+1}$.
In this paper, we will use the curve $c_0$ as figure 1 shows.
Denote $b_i=\sigma^i(b_0), c_i=\sigma^i(c_0)$. They are also used in this paper.

We need the intersection numbers between the curves $a_j, b_k, c_l$.
Consider the index $i,j,k$ in modulo $4g+2$ classes.
When viewing these curves in the $(4g+2)$-gon, we need to be careful.
Some times though two such curves meet at the vertex of the $(4g+2)$-gon,
They do not really intersect. We can perturb them a little to cancel the intersection point.
The intersection numbers between $a_j, b_k, c_l$ are listed as follow:
\begin{enumerate}
  \item $i(a_j,a_k)=0$ if and only if $|j-k| \neq 1$.
  \item $i(a_j,a_k)=1$ if and only if $|j-k| = 1$.
  \item $i(b_j,b_k)=0$ if and only if $|j-k| \not \in \{1,2,3,2g-2,2g\}$.
  \item $i(b_j,b_k)=1$ if and only if $|j-k| \in \{1,3,2g-2,2g\}$.
  \item $i(b_j,b_k)=2$ if and only if $|j-k| = 2$.
  \item $i(c_j,c_k)=0$ if and only if $j = k$.
  \item $i(c_j,c_k)=1$ if and only if $j \neq k$.
  \item $i(a_j,b_k)=0$ if and only if $j-k \not \in \{0,4\}$.
  \item $i(a_j,b_k)=1$ if and only if $j-k \in \{0,4\}$.
  \item $i(a_j,c_k)=0$ if and only if $k-j \not \in \{-1,0\}$.
  \item $i(a_j,c_k)=1$ if and only if $k-j \in \{-1,0\}$.
  \item $i(b_j,c_k)=0$ if and only if $k-j \not \in \{0,1,2,3\}$.
  \item $i(b_j,c_k)=1$ if and only if $k-j \in \{0,1,2,3\}$.
\end{enumerate}

\textbf{Some torsion elements}

Obviously we have $\sigma^{4g+2}=1$.
Take the reflection $\tau$ of the regular $(4g+2)$-gon satisfying $\tau (b_0)=b_0$.
We can check $(\tau B_0)^2=1$. See figure 2.

\begin{figure}[htbp]
\centering
\includegraphics[height=3cm]{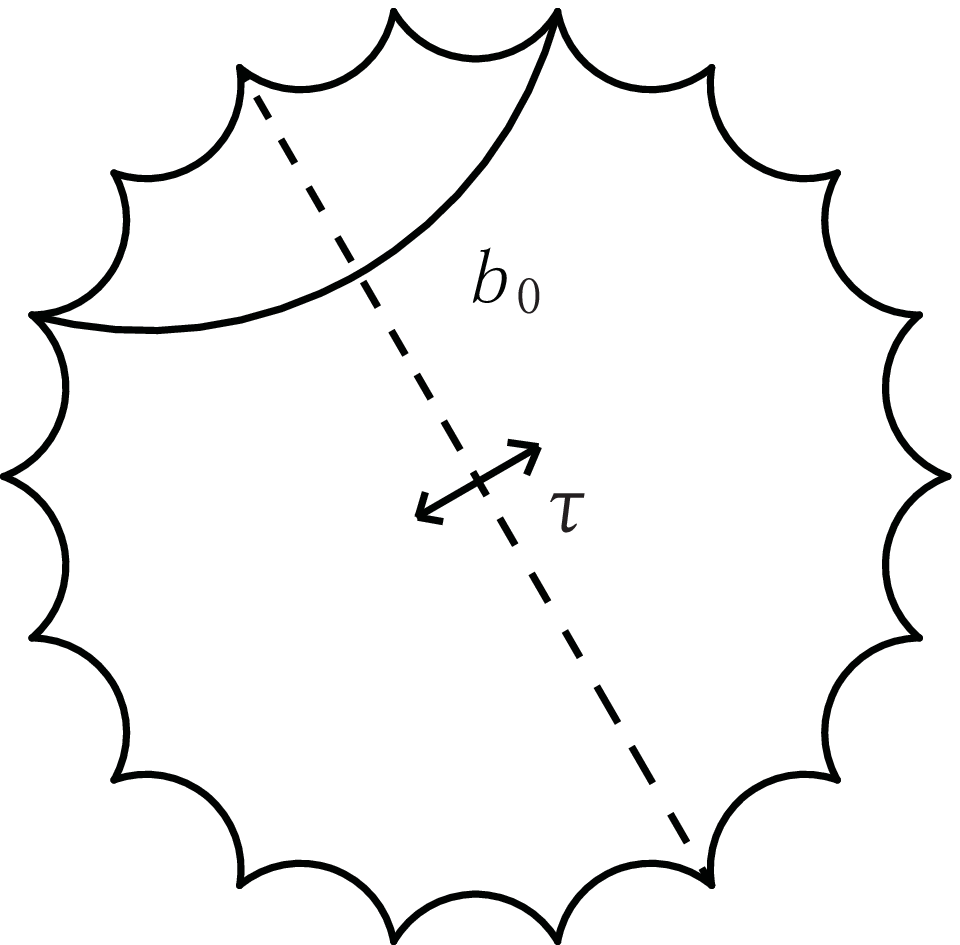} \\
\textbf{Figure 2.}
\end{figure}

In \cite{Du} we know $\text{Mod}^{\pm}(S_g)= \langle \sigma, \tau B_0 \rangle$ for $g \geq 5$.
We will see it is also true for $g=3,4$.

\section{The main result and the proof}

\begin{theorem}
Let $\tau, \sigma, B_0$ as before. For $g=3,4$, $\text{Mod}^{\pm}(S_g) = \langle \sigma, \tau B_0 \rangle$.
\end{theorem}

\begin{proof}

Denote the subgroup generated by $\tau B_0$ and $\sigma$ as $G$.
We will prove that $G$ includes all the elements in $\text{Mod}^{\pm}(S_g)$.
Similar to \cite{Du}, The proof of the lemma has 4 steps.

Step 1. For every $i,k$, we prove $B_iB_k^{-1}$ is in $G$.

Step 2. For every $i,k$, we prove $B_iA_k^{-1}$ is in $G$.

Step 3. Using lantern relation, we prove that for every $i$, $A_i$ is in $G$.

Step 4. $G = \text{Mod}^{\pm}(S_g)$.

The motivation of step 2 and step 3 is as follow.
There is a lantern on the surface where the curves in the lantern relation appear as
$a_1$, $a_3$, $a_5$, $b_0$, $b_2$, $e$, $f$ showed on the upper side of figure 3.
The lantern relation $B_0B_2E=A_1A_3A_5F$ can be also written as $A_1=(B_0A_3^{-1})(B_2A_5^{-1})(EF^{-1})$.
So one Dehn twist can be decomposed into the product of pairs of Dehn twists.
Draw the lantern in the $(4g+2)$-gon as on the lower side of figure 3.
We will find some of the pairs of Dehn twists we use can be expressed as the form $B_k A_i^{-1}$.
When the $g \leq 2$, we cannot find a lantern on the surface.

\begin{figure}[htbp]
\centering
\includegraphics[height=7cm]{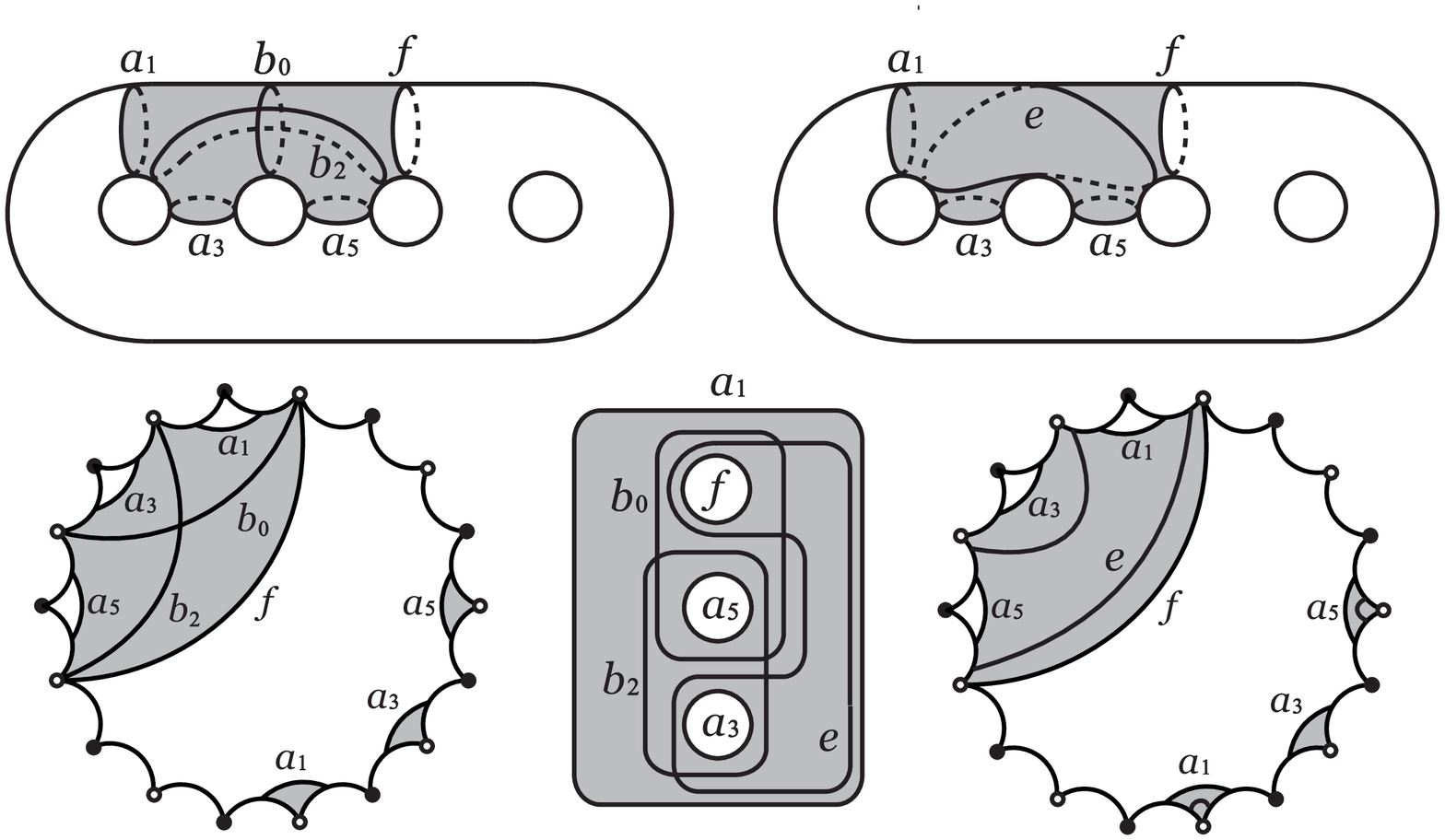} \\
\textbf{Figure 3.}
\end{figure}

\bigskip

The proof of Step 1:

The reason for $B_0B_k^{-1} \in G$ is $\sigma^{k}(\tau B_0)\sigma^{k}(\tau B_0) = B_0B_{k}^{-1}$.
After conjugating with $\sigma^{i}$, we have for every $i,k$, $B_iB_{i+k}^{-1}$ is in $G$.

\bigskip

The proof of step 2:

Suppose the genus $g=4$.

We already know $b_{11}$ does not intersect with $b_0$ or $b_6$.
So $B_{11}B_6^{-1}$ maps the pair of curves $(b_{11},b_0)$ to the pair of curves $(b_{11},B_6^{-1}(b_0))$.
Since $B_{11}B_0^{-1}$ is in $G$, $B_{11}(B_6^{-1}B_0^{-1}B_6)$ is in $G$.
We also have for every $k$,
$B_{k}(B_6^{-1}B_0^{-1}B_6)$ $=(B_{k}B_{11}^{-1})$ $(B_{11}(B_6^{-1}B_0^{-1}B_6))$ is in $G$.
See figure 4.

\begin{figure}[htbp]
\centering
\includegraphics[height=3.5cm]{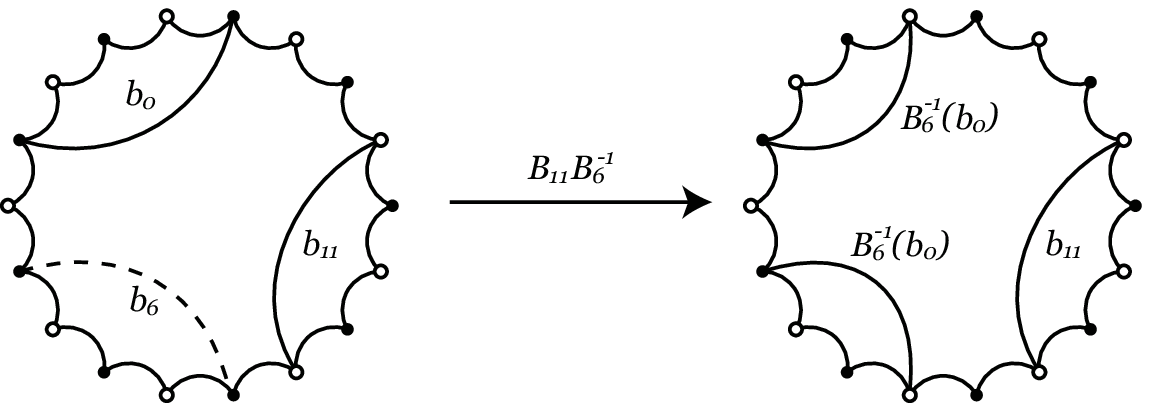} \\
\textbf{Figure 4.}
\end{figure}

We know $b_{1}$ does not intersect with $b_5$. We can check $B_1B_5^{-1}B_6^{-1}(b_0) = a_5$.
So $B_{5}^{-1}$ maps the pair of curves $(b_5,B_6^{-1}(b_0))$ to the pair of curves $(b_5,B_5^{-1}B_6^{-1}(b_0))$,
$B_{1}$ maps the pair of curves $(b_5,B_5^{-1}B_6^{-1}(b_0))$ to the pair of curves $(b_5,a_5)$.
This means $B_1 B_5^{-1}$ maps the pair of curves $(b_5,B_6^{-1}(b_0))$ to the pair of curves $(b_5,a_5)$.
See figure 5.

\begin{figure}[htbp]
\centering
\includegraphics[height=3.5cm]{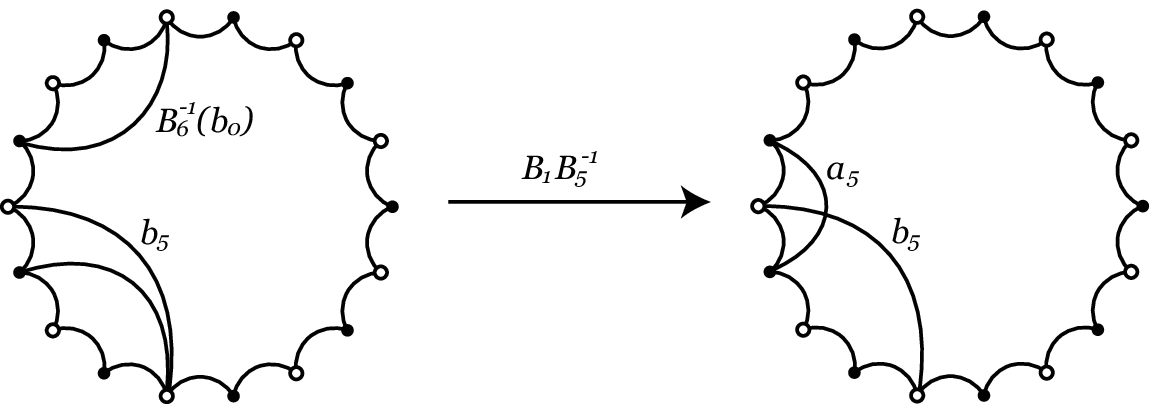} \\
\textbf{Figure 5.}
\end{figure}

Hence $B_5A_5^{-1}$ is in $G$.
After conjugating some power of $\sigma$ and multiplying some $B_iB_j^{-1}$,
we have for every $i,j$, $B_{i}A_j^{-1}$ is in $G$.

\bigskip

Suppose the genus $g=3$.

We know that $b_9$ does not intersect with $b_0$ or $b_4$.
So $B_9 B_4^{-1}$ maps the pair of curves $(b_{9},b_0)$ to the pair of curves $(b_{9},B_4^{-1}(c_0))$.
We can also check when the genus is 3, $c_0 = B_4^{-1}(b_0)$. So $B_9 C_0^{-1}$ is in $G$.
See figure 6.

\begin{figure}[htbp]
\centering
\includegraphics[height=3.5cm]{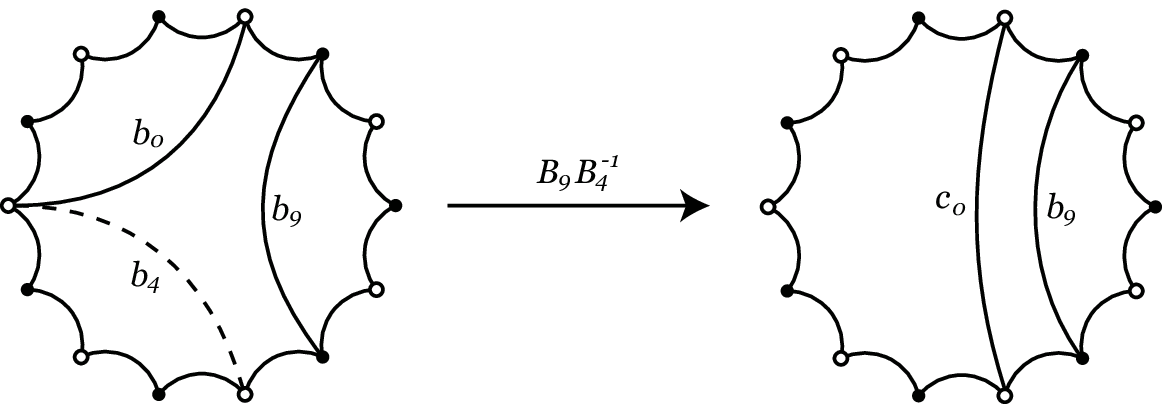} \\
\textbf{Figure 6.}
\end{figure}

After conjugating with some power of $\sigma$ and multiplying some $B_iB_j^{-1}$,
we have for every $i,j$, $B_{i}C_j^{-1}$ and $C_iB_j^{-1}$ are in $G$.
We also have for every $i,j$, $C_{i}C_{j}^{-1}$ is in $G$.

We know $c_0$ does not intersect with $b_1$ or $b_2$.
So $B_{2}C_{0}^{-1}$ maps the pair of curves $(c_{0},b_{1})$ to the pair of curves $(c_{0},B_2(b_1))$.
Then $C_0 (B_2 B_1 B_2^{-1})$ is in $G$. For every $i$, $C_i (B_2 B_1 B_2^{-1})$ is also in $G$.
See figure 7.

\begin{figure}[htbp]
\centering
\includegraphics[height=3.5cm]{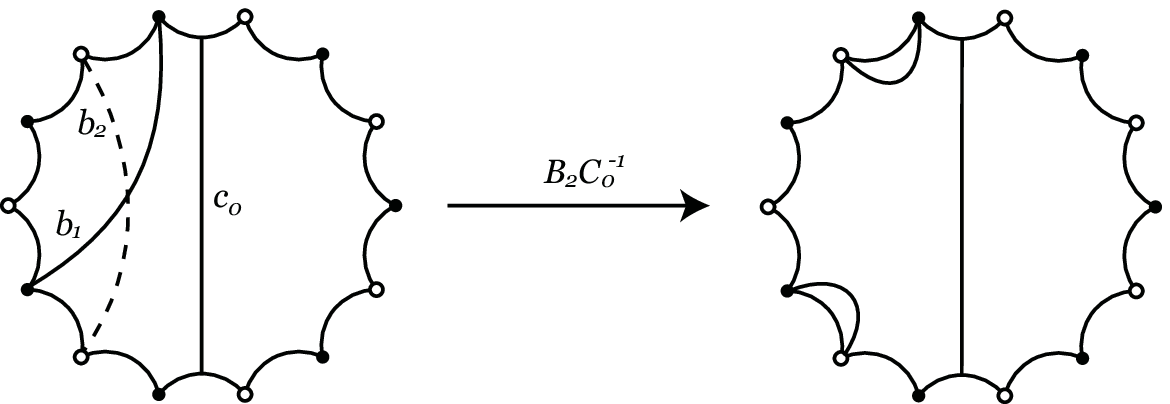} \\
\textbf{Figure 7.}
\end{figure}

We know $c_4$ does not intersect with $b_6$ or $B_2(b_{1})$.
So $C_{4}B_{6}^{-1}$ maps the pair of curves $(c_{4},B_2(b_{1}))$ to the pair of curves $(c_{4},B_6^{-1}B_2(b_1))$.
Then $C_0 (B_6^{-1} B_2 B_1$ $B_2^{-1} B_6)$ is in $G$.
See figure 8.

\begin{figure}[htbp]
\centering
\includegraphics[height=3.5cm]{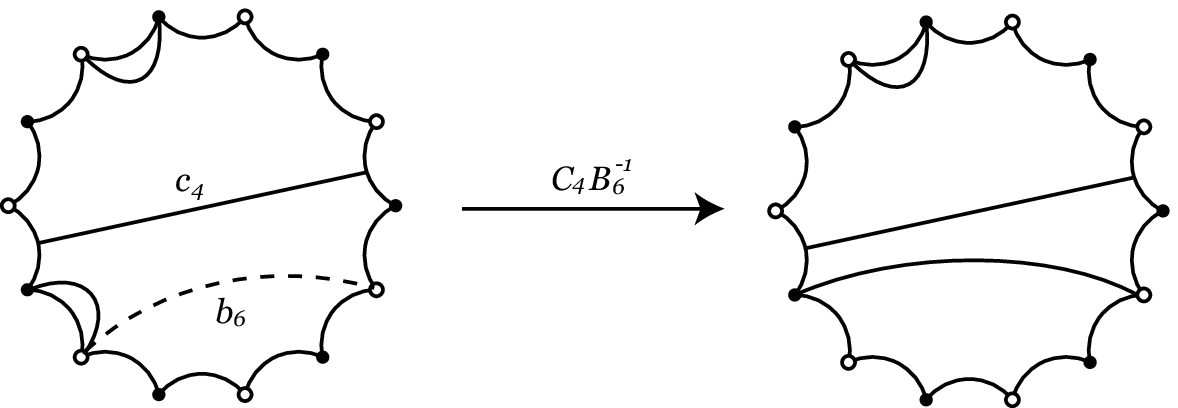} \\
\textbf{Figure 8.}
\end{figure}

We know $c_4$ does not intersect with $b_5$ or $B_6^{-1}B_2(b_1)$.
So $C_{4}B_{5}^{-1}$ maps the pair of curves $(c_{4},B_6^{-1}B_2(b_{1}))$ to the pair of curves $(c_{4},$ $B_5^{-1}$$B_6^{-1}$$B_2(b_1))$.
Then $C_4$ $(B_5^{-1}B_6^{-1}B_2B_1B_2^{-1}B_6B_5)$ is in $G$.
See figure 9.

\begin{figure}[htbp]
\centering
\includegraphics[height=3.5cm]{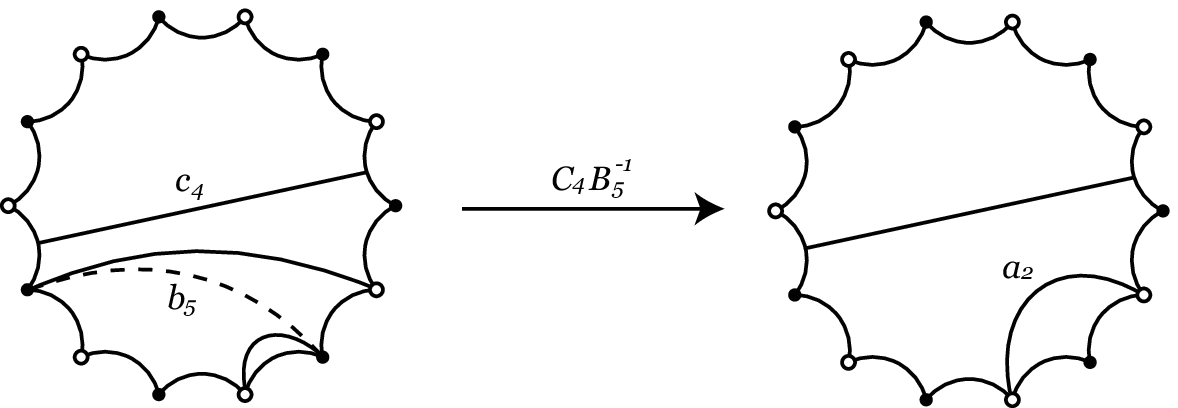} \\
\textbf{Figure 9.}
\end{figure}

We can check that $B_5^{-1}B_6^{-1}B_2(b_1)=a_2$.
So $C_4 A_2^{-1} = C_4$ $(B_5^{-1} B_6^{-1} B_2$ $B_1 B_2^{-1} B_6 B_5)$ is in $G$.
Conjugating with some power of $\sigma$ and multiplying $C_jC_k^{-1}$, we have for every $j,k$, $C_jA_k^{-1}$ is in $G$.
Multiplying it by $B_i C_j^{-1}$, we have for every $i,k$, $B_iA_k^{-1}$ is in $G$.

\bigskip

The proof of step 3:

We want to show for every $i$, $A_i$ is in $G$.

Recall lantern relation, we have $B_0B_2E = A_1A_3A_5F$, or
$A_1=$ $(B_0A_3^{-1})$ $(B_2A_5^{-1})$ $(EF^{-1})$,
where $e$ and $f$ are the curves showed in figure 3.
By the result of step 2, $B_0A_3^{-1}$ and $B_2A_5^{-1}$
are in $G$. What we need is to prove $EF^{-1}$ is also in $G$.
Notice $EF^{-1}=(EB_{i}^{-1})(B_iB_j^{-1})(B_{j}F^{-1})$.
We only need to show there exist some $i,j$ such that
$EB_i^{-1}$ and $B_jF^{-1}$ are in $G$.

Suppose $g=4$.

We can check that $f=B_3^{-1}A_6A_5A_4(b_0)$.
We also know $b_7$ does not intersect with $a_4, a_5, a_6, b_3$.
So $(B_7B_3^{-1})(A_6B_7^{-1})(A_5B_7^{-1})(A_4B_7^{-1})$ maps $(b_7, b_0)$ to $(b_7,f)$.
Hence $B_7F^{-1}$ is in $G$.
See figure 10.

\begin{figure}[htbp]
\centering
\includegraphics[height=3.5cm]{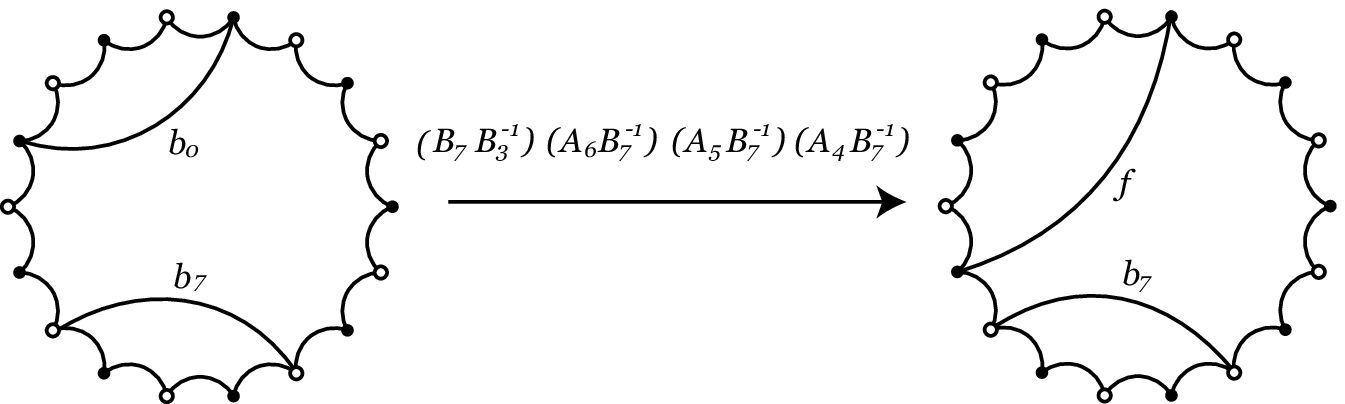} \\
\textbf{Figure 10.}
\end{figure}

We can check $e=A_2A_1A_4^{-1}B_1(a_5)$.
Since $b_{12}$ does not intersect with $a_1, a_2, a_4, a_5, b_1$,
$(A_2B_{12}^{-1})$ $(A_1B_{12}^{-1})$ $(B_{12}A_4^{-1})$ $(B_1B_{12}^{-1})$ maps $(a_5,b_{12})$ to $(e,b_{12})$.
Hence $EB_{12}^{-1}$ is in $G$.
See figure 11.

\begin{figure}[htbp]
\centering
\includegraphics[height=3.5cm]{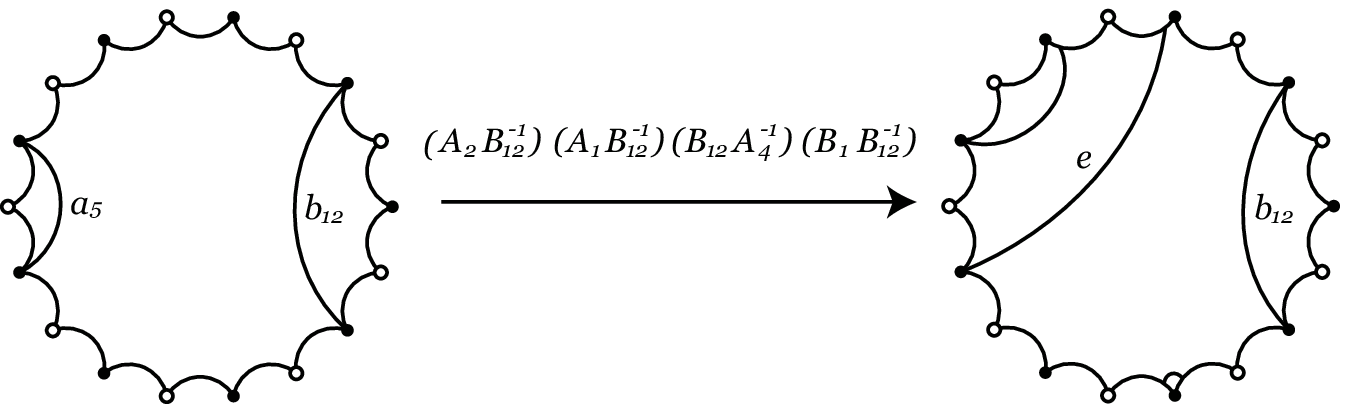} \\
\textbf{Figure 11.}
\end{figure}

\bigskip

Suppose $g=3$.

The fact $f=B_3^{-1}A_6A_5A_4(b_0)$ still holds.
When $g=3$ we cannot find some $b_i$ that does not intersect with $a_4, a_5, a_6, b_3$ simultaneously.
We use some curves $c_i$ instead.

At first we find $c_6$ does not intersect with $a_4, a_5, b_0$.
So $(A_5 C_6^{-1})$ $(A_4 C_6^{-1})$ maps $(c_6,b_0)$ to $(c_6, A_5A_4(b_0))$,
$C_6 (A_5 A_4 B_0 A_4^{-1} A_5^{-1})^{-1}$ is in $G$.
See figure 12.

\begin{figure}[htbp]
\centering
\includegraphics[height=3.5cm]{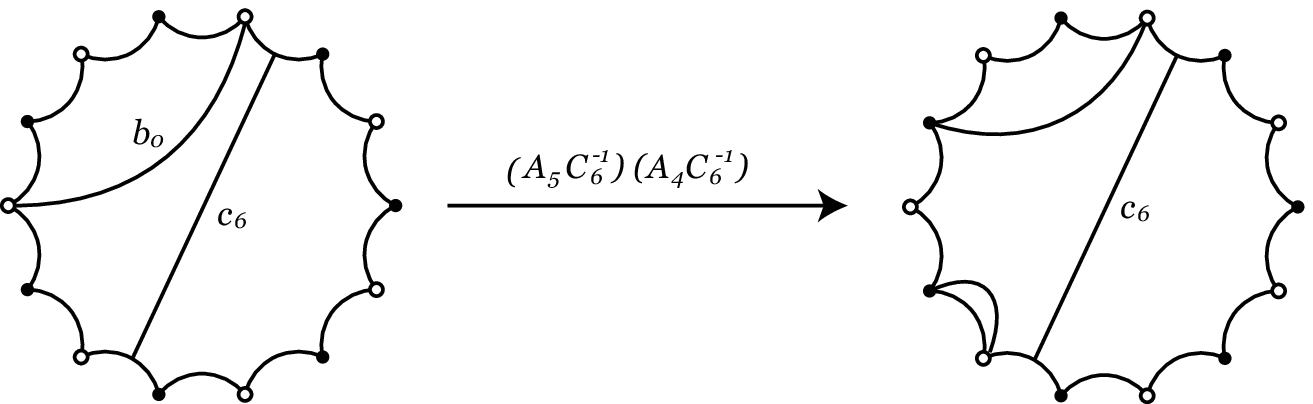} \\
\textbf{Figure 12.}
\end{figure}

$B_8 (A_5 A_4 B_0 A_4^{-1} A_5^{-1})^{-1} = (B_8 C_6^{-1})(C_6 (A_5 A_4 B_0 A_4^{-1} A_5^{-1})^{-1})$ is also in $G$.
Then we find $b_8$ does not intersect with $a_6, b_3$ or $A_5A_4(b_0)$.
So $(B_8 B_3^{-1})(B_8^{-1} A_6)$ maps $(b_8, A_5A_4(b_0))$ to $(b_8, B_3^{-1}A_6A_5A_4(b_0))=(b_8,f)$.
Hence $B_8 F^{-1}$ is in $G$.
See figure 13.

\begin{figure}[htbp]
\centering
\includegraphics[height=3.5cm]{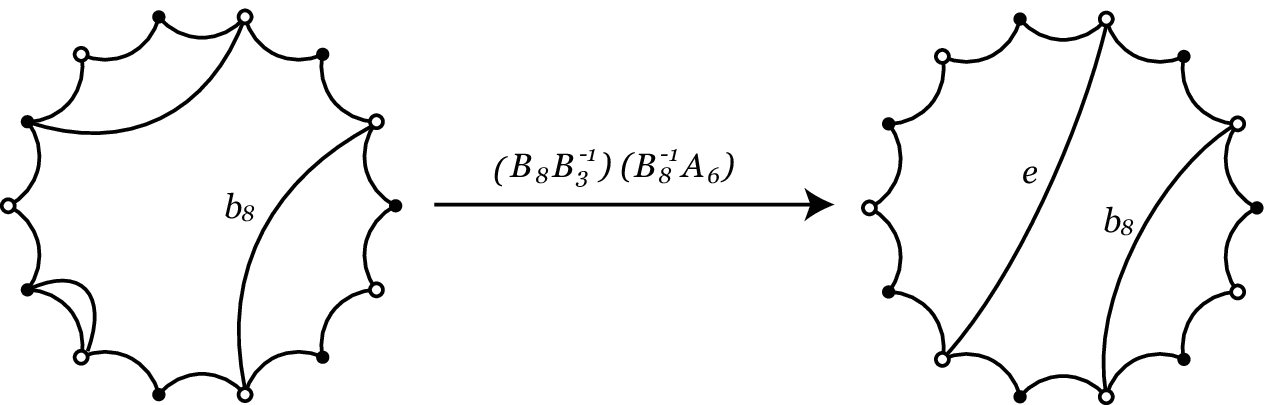} \\
\textbf{Figure 13.}
\end{figure}

Similarly,
The fact $e=A_2A_1A_4^{-1}B_1(a_5)$ still holds.
When $g=3$, we can find $c_i$ does not intersect with $a_1, a_2, a_4, a_5, b_1$.
So $(A_2 C_{6}^{-1})$ $(A_1 C_{6}^{-1})$ $(C_{6} A_4^{-1})$ $(B_1 C_{6}^{-1})$ maps $(a_5,c_{6})$ to $(e,c_{6})$.
Hence $E C_{6}^{-1}$ is in $G$. And then multiply $C_6 B_i^{-1}$, we have $E B_i^{-1}$ in $G$.
See figure 14.

\begin{figure}[htbp]
\centering
\includegraphics[height=3.5cm]{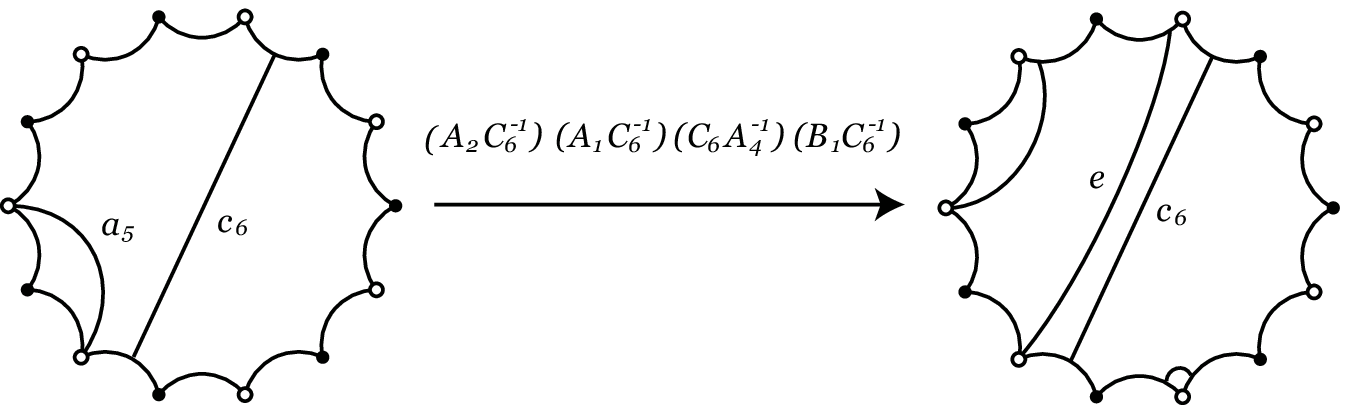} \\
\textbf{Figure 14.}
\end{figure}

\bigskip

The proof of step 4:

Since both $B_iA_j^{-1}$ and $A_j$ are in $G$, by Humphries's result, $G$ contains the mapping class group $\text{Mod}(S_g)$.
Now $\tau B_0 \in G$ is an orientation reversing element. $\text{Mod}(S_g)$ is an index 2 subgroup of $\text{Mod}^{\pm}(S_g)$.
So $G = \text{Mod}^{\pm}(S_g)$ for $g=3,4$.
\end{proof}

\begin{theorem}
For $g=1$, $\text{Mod}^{\pm}(S_1)$ is $GL(2,\mathbb{Z})$. It cannot be generated by two elements of finite orders.
\end{theorem}

\begin{proof}
We only need to prove that $PGL(2,\mathbb{Z})$ cannot be generated by two elements of finite orders.

It is well known that $PGL(2,\mathbb{Z}) \cong D_{6} *_{\mathbb{Z}_2} D_4$,
where $D_6$ and $D_4$ are the dihedral group of order 6 and order 4 respectively
(see, for example, \cite{Koru} or \cite{SIK}). It has a presentation as
$$PGL(2,\mathbb{Z})=\langle a,b,t \mid a^3=t^2=b^2=1, at=ta^2, bt=tb \rangle.$$
Every element $\alpha$ in $PGL(2,\mathbb{Z})$ can be written as a reduced form in one of the following 3 types:
\begin{enumerate}
  \item $a^{i_1}b^{j_1} \dots a^{i_k}b^{j_k}t$,
  \item $b^{j_0}a^{i_1}b^{j_1} \dots a^{i_k}b^{j_k}t$, or
  \item $a^{i_1}b^{j_1} \dots a^{i_k}b^{j_k}a^{j_{k+1}}t$.
\end{enumerate}
Here each $i_n \in \{1,2\}$ and each $j_n = 1$.

For an element in type (2), $b^{j_0}a^{i_1}b^{j_1} \dots a^{i_k}b^{j_k}t$ can be conjugated to\\
$a^{i_1}b^{j_1} \dots a^{i_k}b^{j_k+j_0}t$ $= a^{i_1}b^{j_1} \dots a^{i_k}t$.
So its conjugacy class is the same as an element in type (3) with a shorter word length.

For an element in type (3), $a^{i_1}b^{j_1} \dots a^{i_k}b^{j_k}a^{j_{k+1}}t$ $= a^{i_1}b^{j_1} \dots a^{i_k}b^{j_k}ta^{2j_{k+1}}$
can be conjugated to $a^{i_1+2j_{k+1}}b^{j_1} \dots a^{i_k}b^{j_k}t$.
So its conjugacy class is the same as an element in type (1) or (2) with a shorter word length.

For an element in type (1), since the word length its power will be larger, it must not be of finite order.
So an element of finite order in $PGL(2,\mathbb{Z})$ must be conjugated to one of the following 8 elements:
$1, a, a^2, t, at, a^2t, b, bt$.

By adding in a new relation $ab=ba$, we get a quotient group homomorphic to the
Cartesian product $D_6 \times \mathbb{Z}_2$, which is a finite group with the presentation
$$\langle a_1, t_1, b_1 \mid a_1^3 = t_1^2 = b_1^2 = 1, a_1t_1 = t_1a_1^2, b_1t_1 = t_1b_1, a_1b_1 = b_1a_1 \rangle.$$

For the convenience of calculation, we can think of $D_6 \times \mathbb{Z}_2$ as a permutation subgroup of the symmetric group $S_5$:
\begin{table}[h]
\begin{center}
\begin{tabular}{|c|c|c|c|}
  \hline
  element in $D_6 \times \mathbb{Z}_2$ & permutation & element in $D_6 \times \mathbb{Z}_2$ & permutation\\
  \hline
  1 & () & $b_1$ & (45) \\
  \hline
  $a_1$ & (123) & $a_1b_1$ & (123)(45) \\
  \hline
  $a_1^2$ & (132) & $a_1^2b_1$ & (132)(45) \\
  \hline
  $t_1$ & (12) & $b_1t_1$ & (12)(45) \\
  \hline
  $a_1t_1$ & (13) & $a_1b_1t_1$ & (13)(45) \\
  \hline
  $a_1^2t_1$ & (23) & $a_1^2b_1t_1$ & (23)(45) \\
  \hline
\end{tabular}
\end{center}
\end{table}

We can check all the possible images in $D_6 \times \mathbb{Z}_2$ of the conjugacy classes in $PGL(2,\mathbb{Z})$ as follow:
\begin{table}[h]
\begin{center}
\begin{tabular}{|c|c|}
  \hline
  Conjugacy classes in $PGL(2,\mathbb{Z})$ & elements in $D_6 \times \mathbb{Z}_2$\\
  \hline
  $a$, $a^2$ & $a_1$, $a_1^2$ \\
  \hline
  $t$, $at$, $a^2t$ & $t_1$, $a_1t_1$, $a_1^2t_1$\\
  \hline
  $b$, $bt$ & $b_1$, $b_1t_1$\\
  \hline
\end{tabular}
\end{center}
\end{table}

We have the following fact:
if two elements in $\{a_1, a_1^2, t_1, a_1t_1, a_1^2t_1, b_1, b_1t_1\}$
can generate $D_6 \times \mathbb{Z}_2$, the only possible cases are
$\langle a_1t_1, b_1t_1\rangle$ and $\langle a_1^2t_1, b_1t_1\rangle$.

Each one of $a_1t_1, a_1^2t_1, b_1t_1$ has order 2.
If an element in $PGL(2,\mathbb{Z})$ is mapped to $a_1t_1$, $a_1^2t_1$, or $b_1t_1$,
it must also have order 2. But two elements of order 2 can only generate a dihedral group,
not $PGL(2,\mathbb{Z})$.
\end{proof}

\begin{remark}
Though $\text{GL}(2,\mathbb{Z})$ cannot be generated by two torsion elements, it can be generated by two elements.
In fact, since
$$
\left(
\begin{array}{cc}
0 & 1 \\
1 & 0
\end{array}
\right)
\left(
\begin{array}{cc}
1 & 1 \\
0 & 1
\end{array}
\right)
\left(
\begin{array}{cc}
0 & 1 \\
1 & 0
\end{array}
\right)
=
\left(
\begin{array}{cc}
1 & 0 \\
1 & 1
\end{array}
\right),
$$
we have
$$\text{GL}(2,\mathbb{Z})=\langle
\left(
\begin{array}{cc}
1 & 1 \\
0 & 1
\end{array}
\right),
\left(
\begin{array}{cc}
0 & 1 \\
1 & 0
\end{array}
\right)
\rangle.$$
So the extended mapping class group $GL(2,\mathbb{Z})$ for $g=1$ case:
(1) can be generated by two elements; (2) cannot be generated by two elements of finite orders.
This is different from the case $g \geq 3$. 
\end{remark}

\end{document}